\documentclass[11pt]{amsart}
\usepackage{amssymb,amscd,amsxtra,verbatim}
\usepackage[all]{xy}
\usepackage{pstricks}
\usepackage{pstricks}
\usepackage{pb-diagram}
\usepackage{pstricks,pst-node,pst-tree}

\newtheorem{thm}{Theorem}
\newtheorem{prop}[thm]{Proposition}
\newtheorem{lem}[thm]{Lemma}

\newtheorem{problem}[thm]{Problem}

\theoremstyle{remark}
\newtheorem{remark}{Remark}

\def\wt{{\rm wt}}
\def\inv{{\rm inv}}
\def\T{\mathcal T}
\def\Tab{{\rm Tab}}
\def\Tree{{\rm Tree}}
\def\Perm{{\rm Perm}}
\def\splice{{\rm SG}}
\def\Fib{{\rm Fib}}
\def\qDGG{{\rm qDGG}}
\def\N{{\mathbb N}}

\def\C{{\mathbb C}}

\newlength{\cellsz}
\newcounter{cellsize}
\newcommand{\setcellsize}[1]{%
  \setcounter{cellsize}{#1}%
  \setlength{\cellsz}{\value{cellsize}\unitlength}}%
\setcellsize{13}%
\newcommand\cellify[1]{\def\thearg{#1}\def\nothing{}%
\ifx\thearg\nothing \vrule width0pt height\cellsz depth0pt\else
\hbox to 0pt{{\begin{picture}(\value{cellsize},\value{cellsize})
  \put(0,0){\line(1,0){\value{cellsize}}}
  \put(0,0){\line(0,1){\value{cellsize}}}
  \put(\value{cellsize},0){\line(0,1){\value{cellsize}}}
  \put(0,\value{cellsize}){\line(1,0){\value{cellsize}}} \end{picture} \hss}}\fi%
\vbox to \cellsz{ \vss \hbox to \cellsz{\hss$#1$\hss} \vss}}
\newcommand\tableau[1]{\vcenter{\vbox{\let\\\cr
\baselineskip -16000pt \lineskiplimit 16000pt \lineskip 0pt
\ialign{&\cellify{##}\cr#1\crcr}}}}
\newcommand\tabl[1]{\vtop{\let\\\cr
\baselineskip -16000pt \lineskiplimit 16000pt \lineskip 0pt
\ialign{&\cellify{##}\cr#1\crcr}}}

\title{Quantized dual graded graphs}
 \author{Thomas Lam}\address
 {Department of Mathematics\\ Harvard University\\ Cambridge\\ MA 02138.}
 \date{\today}
 \email{tfylam@math.harvard.edu}
 \urladdr{http://www.math.harvard.edu/\~{ }tfylam}
 \thanks{T.L. was partially supported by NSF grants DMS-0600677 and DMS-0652641.}
\begin{document}
 \begin{abstract}
We study quantized dual graded graphs, which are graphs equipped
with linear operators satisfying the relation $DU - qUD = rI$. We
construct examples based upon: the Fibonacci poset, permutations,
standard Young tableau, and plane binary trees.
 \end{abstract}
\maketitle
\section{Introduction}
Fomin's {\it dual graded graphs} \cite{Fom} and Stanley's {\it
differential posets} \cite{Sta} are constructions developed to
understand and generalize the enumerative consequences of the
Robinson-Schensted algorithm.  The key relation in these
constructions is $DU - UD = rI$, where $U, D$ are up-down operators
acting on the graphs or posets\footnote{Fomin \cite{Fom} also
considered more general relations of the form $DU = f(UD)$.}.  In
this article we develop some of the basic theory of {\it quantized
dual graded graph}, which are equipped with up-down operators $U, D$
satisfying the $q$-Weyl relation $DU - qUD = rI$.  One of the
motivations for the current work were the signed differential posets
developed in \cite{Lam}, which correspond to the relation $DU + UD =
rI$. Thus quantized dual graded graphs specialize to usual dual
graded graphs at $q = 1$, and to signed differential posets (or
their dual graded graph equivalent) at $q = -1$.

The central enumerative identity in the subject developed by Fomin
and Stanley is
$$
\sum_{\lambda \vdash n} (f^\lambda)^2 = n!
$$
where the sum is over partitions of $n$, and $f^\lambda$ is the
number of standard Young tableau of shape $\lambda$.  The
corresponding analogue (Theorem \ref{thm:main}) for a quantized dual
graded graph $(\Gamma, \Gamma')$ reads
\begin{equation}\label{eq:main}
\sum_{v} f_\Gamma^v(q) \, f_{\Gamma'}^v(q) = r^n [n]_q!
\end{equation}
where the sum is over vertices of height $n$, the polynomials
$f_\Gamma^v(q)$ and $f_{\Gamma'}^v(q)$ are weighted enumerations of
paths in $\Gamma$ and $\Gamma'$, and $[n]_q!$ is the $q$-analogue of
$n!$.

We explicitly construct examples of quantized dual graded graphs and
interpret \eqref{eq:main}.  These examples are based on various
combinatorial objects: the Fibonacci poset, permutations, standard
Young tableau, and plane binary trees. Unfortunately, we have been
unable to quantize Young's lattice.  More examples will be given in
joint work \cite{BLL} with Bergeron and Li, where in some cases a
representation theoretic explanation for the identities $DU - qUD =
I$ and \eqref{eq:main} will be given.

\section{Quantized dual graded graphs}
Let $\Gamma = (V,E,m)$ be a graded graph with edge weights $m(v,w)
\in \N[q]$. That is, $\Gamma$ is a directed graph with a height
function $h: V \to \N$ such that if $(v,w) \in E$ then $h(w) = h(v)
+ 1$.  Furthermore, each edge has a weight $m(v,w) \in \N[q]$ which
is a non-zero polynomial in $q$ with nonnegative coefficients.  We
shall assume that $\Gamma$ is locally finite, so that for each $v$,
there are finitely many edges entering and leaving.  Because each
edge has a weight, we shall assume that there are no multiple edges.

Let $\widehat{\C(q)[V]}$ be the $\C(q)$-vector space of formal
linear combinations of the vertex set $V$.  A linear operator on
$\widehat{\C(q)[V]}$ is {\it continuous} if it is compatible with
arbitrary linear combinations.  Define continuous linear operators
$U, D: \widehat{\C(q)[V]} \to \widehat{\C(q)[V]}$ by
\begin{align*}
U(v) &= \sum_{w: (v,w) \in E} m(v,w)\, w \\
D(w) &= \sum_{v: (v,w) \in E} m(v,w)\, v.
\end{align*}
and extending by linearity and continuity.  We define a pairing
$(.,.): \widehat{\C(q)[V]} \times \C(q)[V] \to \C(q)$ by $(v,w) =
\delta_{v,w}$ for $v, w \in E$.  Then $U$ and $D$ are adjoint with
respect to this pairing.

Let $(\Gamma = (V,E,m), \Gamma' = (V, E', m'))$ be a pair of graded
graphs with the same vertex set.  Then $(\Gamma,\Gamma')$ is a pair
of {\it quantized dual graded graphs} ($\qDGG$ for short) if we have
the identity
\begin{equation}\label{eq:qDGG}
D_{\Gamma'}U_{\Gamma} - qU_\Gamma D_{\Gamma'} = rI
\end{equation}
for some positive integer $r \in \{1,2,3,\ldots\}$, called the {\it
differential coefficient}.  In the sequel, we will often write $U$
and $D$ for $U_\Gamma$ and $D_{\Gamma'}$. When $q = 1$, we obtain
the dual graded graphs of \cite{Fom}, which are equipped with the
relation $DU - UD = rI$.  We should note that Fomin also considered
the more general relation $DU = f(UD)$ for arbitrary functions $f$;
however, he did not focus on \eqref{eq:qDGG} where $q$ is a
parameter.

If $(\Gamma(q),\Gamma'(q))$ are a pair of quantized dual graded
graphs then we say that $(\Gamma(q),\Gamma'(q))$ is a quantization
of $(\Gamma(1),\Gamma'(1))$.  The basic example of a dual graded
graph is Young's lattice of partitions, ordered by containment; see
\cite{Fom, Sta}.  The following is the basic problem for quantized
dual graded graphs.

\begin{problem}
Find a quantization of Young's lattice.
\end{problem}

In \cite{LS}, we constructed dual graded graphs from the strong
(Bruhat) and weak orders of the Weyl group of a Kac-Moody algebra.
The dual graded graphs constructed this way include Young's lattice,
and closely related graphs such as the shifted Young's lattice.

\begin{problem}
Find a quantization of Kac-Moody dual graded graphs.
\end{problem}

\begin{remark}
Equation \eqref{eq:qDGG} specializes to $DU + UD = I$ when $q = -1$
and $r = 1$.  Graphs satisfying this relation were studied in
\cite{Lam}.  More specifically, in \cite{Lam} we studied only such
graphs, called signed differential posets, which arose from labeled
posets. The examples constructed in the present paper can also be
specialized at $q = -1$, giving what would be called ``signed dual
graded graphs''.  The main example in \cite{Lam} was the
construction of a signed differential poset structure on Young's
lattice.  Since we have been unable to quantize Young's lattice, we
have stopped short of explicitly writing the examples in the current
article using the notation in \cite{Lam}.
\end{remark}

\section{$q$-derivatives and enumeration on quantized dual graded graphs}
Let $f(t) = \sum_{n \geq 0} a_n t^n \in \C[[t]]$ be a formal power
series in one variable. Define the {\it $q$-derivative} as follows:
\begin{align*}
f^{q}(t) &= \sum_{n \geq 0} [n]_q a_n t^{n-1}.
\end{align*}
Here $[n]_q = 1 + q + \ldots + q^{n-1}$ denotes the $q$-analogue of
$n$.  We also set $[n]_q! := [n]_q [n-1]_q \cdots [2]_q [1]_q$.  Let
$U, D$ be formal, non-commuting variables satisfying the relation
$DU - qUD = r$.  We assume that $U$ and $D$ commute with the
variable $q$.  The following Lemma explains the relationship between
the relation $DU-qUD=r$ and $q$-derivatives.
\begin{lem}
\label{lem:diff} Suppose $f(U) \in \C[[U]]$ is a formal power series
in the variable $U$.  Then $Df(U) = r\, f^q(U) + f(qU)D$.
\end{lem}
\begin{proof}
By linearity and continuity it suffices to prove the statement for
$f(U) = U^n$.  For $n = 0$, the formula is trivially true.  The
inductive step follows from the calculation
\begin{align*}
DU^n &= (r [n-1]_qU^{n-2} + q^{n-1}U^{n-1}D)U \\
&= r [n-1]_qU^{n-1} + q^{n-1}U^{n-1}(r + qUD) \\
&=r [n]_qU^{n-1} + q^nU^nD,
\end{align*}
using $[n]_q = [n-1]_q + q^{n-1}$.
\end{proof}


We now suppose $(\Gamma,\Gamma')$ is a $\qDGG$ with a unique minimum
(source) $\emptyset$, which we assume has height $h(\emptyset) = 0$.
Let us denote the weight generating function of paths in $\Gamma$
from $\emptyset$ to a vertex $v \in V$ by $f_\Gamma^v = (U^n
\emptyset, v)$, where $n = h(v)$.  It is not difficult to see that
we have
$$
(D^n U^n \emptyset , \emptyset) = \sum_{v:\; h(v) = n} f_\Gamma^v \,
f_{\Gamma'}^v.
$$
The following is an analogue of \cite[Corollary 1.5.4]{Fom}; see
also \cite{Sta}.
\begin{thm}\label{thm:main}
Let $(\Gamma,\Gamma')$ be a $\qDGG$ with a unique minimum
$\emptyset$. Then
$$
\sum_{v: \; h(v) = n} f_\Gamma^v \, f_{\Gamma'}^v = r^n [n]_q!.
$$
\end{thm}
\begin{proof}
By Lemma \ref{lem:diff} we have
\begin{align*}
D^n U^n \, \emptyset= D^{n-1} (r[n]_qU^{n-1} + q^nU^nD)\, \emptyset
= r[n]_q D^{n-1}U^{n-1} \emptyset
\end{align*}
from which the result follows by induction.
\end{proof}

More generally, let $f(\emptyset \to v \to w)$ denote the weight
generating function of paths beginning at $\emptyset$, going up to
$v$ in $\Gamma$, then going down to $w$ in $\Gamma'$.  For fixed $w$
with $h(w) = m \leq n$, we then have
$$
\sum_{v: \; h(v) = n} f(\emptyset \to v \to w) = (D^{n-m}U^n \,
\emptyset,w) = r^{n-m}([n]_q [n-1]_q \cdots [m+1]_q) \,f^w_\Gamma.
$$
Other path generating function problems can be solved by studying
the ``normal ordering problem'' for the relation $DU - qUD = r$,
that is, the problem of rewriting a word in the letters $U$ and $D$
as a linear combination of terms $U^iD^j$.  We shall not pursue this
direction here, but see for example \cite{Var}.

\section{$q$-reflection}
Let $(\Gamma_n = (V,E,m),\Gamma'_n = (V,E',m'))$ be a pair of graded
graphs with height function taking values in $[0,n]$, and such that
\eqref{eq:qDGG} holds for some fixed $r$, when applied to all
vertices $v$ such that $h(v) < n$.  We call such a pair a {\it
partial $\qDGG$} of height $n$. We will construct a partial $\qDGG$
$(\Gamma_{n+1},\Gamma'_{n+1})$ of height $n+1$, and such that they
agree with $(\Gamma_n,\Gamma'_n)$ up to height $n$.

Let us write $V_i = \{v \mid h(v) = i\}$.  The height $n+1$ vertices
of (both) $\Gamma_{n+1}$ and $\Gamma'_{n+1}$ will be given by the
set $V_{n+1} = \{v^1, v^2 ,\ldots, v^r \mid v \in V_n\} \cup \{w'
\mid w \in V_{n-1}\}$.  There will be two kinds of edges.  For
$\Gamma_{n+1}$, we construct
\begin{enumerate}
\item
$r$ edges $(v,v^{1})$, $(v,v^2), \ldots, (v,v^r)$ for each $v \in
V_n$ which have weight $m(v,v'):= 1$.
\item
An edge $(v,w')$ for each edge $(w,v)$ of $\Gamma'$, where $v \in
V_n$ and $w \in V_{n-1}$.  This edge has weight $m(v,w'):=
q\,m'(w,v)$.
\end{enumerate}
And for $\Gamma'_{n+1}$, we construct
\begin{enumerate}
\item
$r$ edges $(v,v^{1})$, $(v,v^2), \ldots, (v,v^r)$ for each $v \in
V_n$ which have weight $m'(v,v'):= 1$.
\item
An edge $(v,w')$ for each edge $(w,v)$ of $\Gamma'$, where $v \in
V_n$ and $w \in V_{n-1}$.  This edge has weight $m'(v,w'):= m(w,v)$.
\end{enumerate}

We omit the proof of the following, which is the same as the
corresponding result for differential posets \cite{Sta} or signed
differential posets \cite{Lam}.
\begin{prop}\label{prop:refl}
Suppose $(\Gamma_n,\Gamma'_n)$ is a partial $\qDGG$ of height $n$
and differential coefficient $r$. Then
$(\Gamma_{n+1},\Gamma'_{n+1})$ is a partial $\qDGG$ of height $n+1$
and differential coefficient $r$. Furthermore, $(\Gamma,\Gamma') =
\lim_{n \to \infty} (\Gamma_n,\Gamma'_n)$ is a well-defined $\qDGG$
with differential coefficient $r$.
\end{prop}

\section{The quantized Fibonacci poset}
Let $(\Gamma, \Gamma')$ be a $\qDGG$.  If the edge sets of $\Gamma$
and of $\Gamma'$ are identical and in addition every edge weight
$m(v,w)$ (and $m'(v,w)$) of $\Gamma$ (and $\Gamma'$) is a single
power $q^i$ then we call $(\Gamma, \Gamma')$ a {\it quantized
differential poset}.  For then, $\Gamma(1)$ would be a differential
poset in the sense of Stanley \cite{Sta}.

\begin{remark}
We could insist that $\Gamma = \Gamma'$ as graded graphs, but then
in the construction of a quantization of the Fibonacci differential
posets we would need to use half powers of $q$.
\end{remark}

Define $\Gamma_0 = \Gamma'_0$ to be the graded graph with a single
vertex $\emptyset$ with height 0. For each $r \in \{1,2,\ldots\}$,
we define the {\it quantized $r$-Fibonacci poset} to be the
corresponding $\qDGG$ $(\Fib_{(r)},\Fib'_{(r)})$ obtained from
$(\Gamma_0,\Gamma'_0)$ via Proposition \ref{prop:refl}.  The $\qDGG$
$(\Fib_{(r)},\Fib'_{(r)})$ is a quantization of the Fibonacci
differential poset of Stanley \cite{Sta}, or the Young-Fibonacci
graph of Fomin \cite{Fom}.  We now describe
$(\Fib_{(r)},\Fib'_{(r)})$ explicitly, suppressing the parameter $r$
in most of the notation.

The vertex set $V$ of $(\Fib_{(r)},\Fib'_{(r)})$ consists of words
$w$ in the letters $1_1, 1_2, \ldots, 1_r, 2$ with height function
given by summing the letters in the word (all the $1$'s have the
same value).  In the notation of the $q$-reflection algorithm, the
vertices $v^1,\ldots, v^r$ are obtained from $v$ by prepending $1_1,
1_2, \ldots, 1_r$ respectively; the vertices $w'$ are obtained from
$w$ by prepending the letter $2$.  The edges $(v,w)$ are of one of
the two forms:
\begin{enumerate}
\item
$v$ is obtained from $w$ by removing the first $1$ (one of the
letters $1_1, 1_2, \ldots, 1_r$);
\item
$v$ is obtained from $w$ by changing a $2$ to one of the $1$'s, such
that all letters to the left of this $2$ is also a $2$.
\end{enumerate}
In either case, let $s(v,w)$ denote the number of letters preceding
the letter which is changed or removed to go from $w$ to $v$.  The
edges $m(v,w)$ of form (1) have edge weight $m(v,w) = m'(v,w)
q^{s(v,w)}$ in both $\Fib_{(r)}$ and $\Fib'_{(r)}$.  The edges of
form (2) have edge weight $m(v,w) = q^{s(v,w)+1}$ in $\Fib_{(r)}$,
and edge weight $m'(v,w) = q^{s(v,w)}$ in $\Fib'_{(r)}$.

For the rest of this section, we will restrict ourselves to $r = 1$,
and write $1$ instead of $1_1$.  We now describe the weight of a
path from $\emptyset$ to a word $w$ in $\Fib = \Fib_{(1)}$ or $\Fib'
= \Fib'_{(1)}$. Given a word $w \in \Fib$ one has a {\it snakeshape}
(\cite{Fom}) consisting of a series of columns of height one or two.
For example, for $w = 21121$ we have the shape
$$ \tableau{{}&&&{}&\\{}&{}&{}&{}&{}}.$$

Given such a {\it snakeshape} $\lambda$, following Fomin \cite{Fom}
we say that a Young-Fibonacci-tableau of shape $\lambda$ is a
bijective filling of $\lambda$ with the numbers $\{1,2\ldots,n\}$ so
that:
\begin{enumerate}
\item
In any height two column the lower number is smaller.
\item
To the right of a height two column containing the numbers $a$ and
$b$ none of the numbers in $[a,b]$ occur.
\item
To the right of a height one column containing the number $a$, no
numbers greater than $a$ occur.
\end{enumerate}
For each number $i \in \{1,2\ldots,n\}$, let $p_i(T)$ denote the
position of the column containing $i$ in $T$, counting from the left
with the leftmost column being 0.  Then set
\begin{align*}
\wt(T) = \prod_{i \in {\rm lower \,\,row}} q^{p_i(T)}  \prod_{i \in
{\rm upper \,\,row}} q^{p_i(T) + 1} \ \ \ \ {\rm and} \ \ \ \
\wt'(T) = \prod_{i} q^{p_i(T)}.
\end{align*}

Fomin \cite{Fom} described a bijection between
Young-Fibonacci-tableau $T$ of shape $\lambda = \lambda(w)$ and
paths from $\emptyset$ to $w$ in $\Fib$ (or $\Fib'$).  For example,
the tableau
$$ \tableau{{3}&&&{5}&\\{2}&{7}&{6}&{4}&{1}}$$
corresponds to the path $\emptyset \to 1 \to 11 \to 21 \to 211 \to
221 \to 2121 \to 21121$.
\begin{lem}
Under this bijection the weight of path is equal to $\wt(T)$ in
$\Fib$, and equal to $\wt'(T)$ in $\Fib'$.
\end{lem}
\begin{proof}
This is straightforward, using the description of the bijection on
\cite[p.394]{Fom}.
\end{proof}
Thus we have $f_\Fib^w = \sum_T \wt(T)$ and $f_{\Fib'}^w = \sum_T
\wt'(T)$ where the sum is over Young-Fibonacci tableau with shape
$\lambda(w)$. It is not clear whether there is a simple way to write
the identity that results from Theorem \ref{thm:main}.

\section{The $\qDGG$ on permutations}
Let $V = \sqcup_{n\geq 0} S_n$ be the disjoint union of all
permutations equipped with the height function $h(w) = n$ if $w \in
S_n$.  Define a graded graph $\Perm$ with vertex set $V$ and edge
set $E$ consisting of edges $(v,w)$ whenever $v \in S_{n-1}$ is
obtained from $w \in S_n$ by deleting the letter $n$; define
$m(v,w):= q^{n-s}$, where $1 \leq s \leq n$ is the position of the
letter $n$ in $w$. Define $\Perm'$ with the same vertex set and
edges $(v,w)$ whenever $v \in S_{n-1}$ is obtained from $w \in S_n$
by deleting the first letter, followed by reducing all letters
greater than the deleted letter by one; define $m(v,w):= 1$ always.

For example, in $\Perm$ there is an edge from $4123$ to $41523$ with
weight $q^3$.  In $\Perm'$ there is an edge from $1423$ to $41523$
with weight $1$.  The following result is a straightforward
verification of the definitions.
\begin{prop}
The pair $(\Perm,\Perm')$ is a $\qDGG$ with differential coefficient
$r = 1$.
\end{prop}

Let $\inv(w)$ denote the number of inversions of a permutation $w$.
For the pair $(\Perm,\Perm')$, we have
$$
f^w_\Perm = q^{\inv(w)}  \ \ \ \ \ {\rm and} \ \ \ \ \  f^w_{\Perm'}
= 1.
$$
Thus Theorem \ref{thm:main} expresses the identity (see \cite{EC1})
\begin{equation} \label{eq:inv}
\sum_{w \in S_n} q^{\inv(w)} = [n]_q!.
\end{equation}

\section{The $\qDGG$ on tableaux}
Let $Y_n$ denote the set of standard Young tableau $P$ of size $n$
with any shape (see \cite{EC2}).  We assume the reader is familiar
with tableaux, and with Schensted insertion.

Let $V = \cup_{i \geq 0} Y_i$ with the obvious height function.
Define $\Tab$ to be the graded graph with vertex set $V$, and edges
$(P,P') \in Y_n \times Y_{n+1}$ whenever there is some $k \in
\{1,2,\ldots,n+1\}$ so that $P'$ is obtained from $P$ by first
increasing the numbers greater than or equal to $k$ inside $P$ by
$1$, and then Schensted inserting $k$; declare $m(P,P') =
q^{n+1-k}$. Define $\Tab'$ to be the graded graph with vertex set
$V$ and edges $(P,P') \in Y_n \times Y_{n+1}$ whenever $P'$ is
obtained from $P$ by removing $n$; declare $m(P,P') = 1$. The
following result is straightforward.

\begin{prop}\label{prop:tab}
The pair $(\Tab,\Tab')$ is a $\qDGG$ with differential coefficient
$r = 1$.
\end{prop}

Fix a standard Young tableau $P \in Y_n$.  There is a bijection from
the set of paths $p$ from the empty tableau $\emptyset$ to $P$ in
$\Tab$, to the set of standard Young tableau of shape equal to the
shape of $T$.  The bijection is obtained by taking the sequence of
shapes encountered along $p$, or equivalently, by taking the
recording tableau of the sequence of Schensted insertions given by
$p$.  The following Lemma is immediate.

\begin{lem}
Suppose $p$ is a path from $\emptyset$ to $P$, corresponding to a
standard Young tableau $Q$.  Then the weight of $p$ in $\Tab$ is
equal to $q^{\inv(w(P,Q))}$, where $w(P,Q) \Leftrightarrow (P,Q)$
under the Robinson-Schensted bijection.
\end{lem}

It follows that Theorem \ref{thm:main} applied to Proposition
\ref{prop:tab} gives \eqref{eq:inv}, with the terms labeled by
permutations $w$ on the left hand side grouped according to the
insertion tableau of $w$.

\section{The $\qDGG$ on plane binary trees}
A {\it plane binary tree} is a tree $T$ embedded into the plane
which has three kinds of vertices: (a) a unique {\it root node} $r$
which has exactly 1 child, (b) a number of {\it internal nodes} with
two children, and (c) a number of {\it leaves} with no children. The
leaves are numbered $\{0,1,\ldots,n\}$ from left to right, where $n$
is the number of internal nodes.  Let $\T_n$ denote the set of plane
binary trees with $n$ internal nodes.  By definition, $\T_0$
consists of the tree $\emptyset$, which has a root $r$, no internal
nodes, and a single leaf $0$.

We now describe a number of combinatorial operations on plane binary
trees; see \cite{AS} for further details.  Given two plane binary
trees $T_1 \in \T_p$ and $T_2 \in \T_q$ we can {\it graft} a new
plane binary tree $T_1 \vee T_2 \in \T_{p+q+1}$ by placing $T_1$ to
the left of $T_2$ in the plane, identifying the two root nodes $r_1$
and $r_2$ to form a new internal node, and attaching a new root to
this internal node:
$$
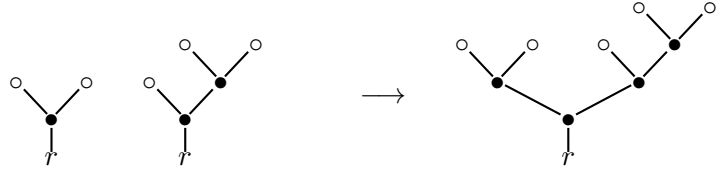
\pstree[treefit=tight,nodesep=0.2pt,levelsep=0.5cm,treemode=U]{\TR{r}}{
    \pstree{\TR{\bullet}}{
            \TR{\circ}\TR{\circ}
    }
} \ \ \ \ \
\pstree[treefit=tight,nodesep=0.2pt,levelsep=0.5cm,treemode=U]{\TR{r}}{
    \pstree{\TR{\bullet}}{
            \TR{\circ}
            \pstree{\TR{\bullet}}{
                \TR{\circ}
                \TR{\circ}
            }
    }
}
\ \ \ \ \ \ \ \ \ \ \raisebox{20pt}{$\longrightarrow$} \ \ \ \ \
\pstree[treefit=tight,nodesep=0.2pt,levelsep=0.5cm,treemode=U]{\TR{r}}{
    \pstree{\TR{\bullet}}{
        \pstree{\TR{\bullet}}{
            \TR{\circ}\TR{\circ}
        }
        \pstree{\TR{\bullet}}{
            \TR{\circ}
            \pstree{\TR{\bullet}}{
                \TR{\circ}
                \TR{\circ}
            }
        }
    }
}
$$
Given a tree $T \in \T_p$ and a position $i \in \{0,1,\ldots,p\}$
indexing a leaf $v \in T$ we can {\it splice} $T$ at $v$ to obtain
two trees $T_1 \in \T_i$ and $T_2 \in \T_p-i$ as follows: draw the
unique path $P$ from $v$ to the root $r$.  Then the edges of $T$
weakly to the left of $P$ form the tree $T_1$, while the edges of
$T$ weakly to the right of $P$ form the tree $T_2$.  Note that every
internal node of $T$ is ``given'' to either $T_1$ or $T_2$.  The
following tree has been spliced at the $*$-ed leaf:
$$
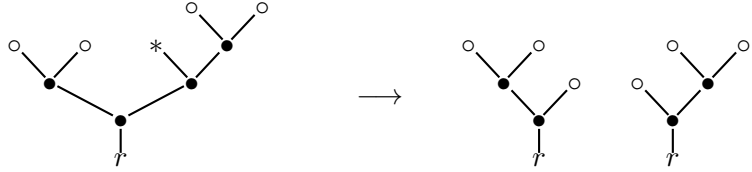
\pstree[treefit=tight,nodesep=0.2pt,levelsep=0.5cm,treemode=U]{\TR{r}}{
    \pstree{\TR{\bullet}}{
        \pstree{\TR{\bullet}}{
            \TR{\circ}\TR{\circ}
        }
        \pstree{\TR{\bullet}}{
            \TR{*}
            \pstree{\TR{\bullet}}{
                \TR{\circ}
                \TR{\circ}
            }
        }
    }
} \ \ \ \ \ \ \ \  \ \raisebox{20pt}{$\longrightarrow$} \ \ \ \ \ \
\pstree[treefit=tight,nodesep=0.2pt,levelsep=0.5cm,treemode=U]{\TR{r}}{
    \pstree{\TR{\bullet}}{
            \pstree{\TR{\bullet}}{
                \TR{\circ}
                \TR{\circ}
            }
            \TR{\circ}
    }
} \ \ \ \ \
\pstree[treefit=tight,nodesep=0.2pt,levelsep=0.5cm,treemode=U]{\TR{r}}{
    \pstree{\TR{\bullet}}{
            \TR{\circ}
            \pstree{\TR{\bullet}}{
                \TR{\circ}
                \TR{\circ}
            }
    }
}
$$
We write $\splice(T,i) = T_1 \vee T_2$ to denote the composition of
splicing and grafting.

Given a non-empty tree $T \in \T_p$, we can obtain another tree $T^*
\in \T_{p-1}$ from $T$ by removing the leftmost (or $0$) leaf $v$
and erasing the node $w$ which is joined to $v$:
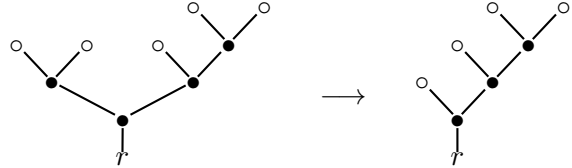
$$
\pstree[treefit=tight,nodesep=0.2pt,levelsep=0.5cm,treemode=U]{\TR{r}}{
    \pstree{\TR{\bullet}}{
        \pstree{\TR{\bullet}}{
            \TR{\circ}\TR{\circ}
        }
        \pstree{\TR{\bullet}}{
            \TR{\circ}
            \pstree{\TR{\bullet}}{
                \TR{\circ}
                \TR{\circ}
            }
        }
    }
} \ \ \ \ \ \raisebox{20pt}{$\longrightarrow$} \ \ \ \ \
\pstree[treefit=tight,nodesep=0.2pt,levelsep=0.5cm,treemode=U]{\TR{r}}{
    \pstree{\TR{\bullet}}{
        \TR{\circ}
        \pstree{\TR{\bullet}}{
            \TR{\circ}
            \pstree{\TR{\bullet}}{
                \TR{\circ}
                \TR{\circ}
            }
        }
    }
}
$$

Let $V = \cup_{i \geq 0} T_i$, with the obvious height function $h:
V \to \N$.  Define a graded graph $\Tree$ with vertex set $V$, and
edges $(T,T')$ whenever $T' = \splice(T,i)$ for some $i$; declare
that $m(T',T):=q^{i}$.  Define a graded graph $\Tree'$ with vertex
set $V$, and edges $(T^*, T)$ for every $T \neq \emptyset$; declare
that $m(T^*,T):=1$.

\begin{prop}
The pair $(\Tree, \Tree')$ is a $\qDGG$ with differential
coefficient $r = 1$.
\end{prop}

\begin{proof}
Let $T \in \T_p$.  Let $T' = \splice(T,i)$, where $i \in
\{1,2,\ldots,p\}$. Then it is not difficult to see that $(T')^* =
\splice(T^*,i-1)$.  This cancels out all the terms in
$(D_{\Tree'}U_{\Tree} - qU_{\Tree}D_{\Tree'})T$ except for the one
corresponding to $\splice(T,0)^*  = T$ which has coefficient $q^0 =
1$.
\end{proof}

To describe the identity of Theorem \ref{thm:main} explicitly, let
us define a linear extension of $T \in \T_p$ to be a bijective
labeling $e: T \to \{1,2,\ldots,p\}$ of the internal nodes of $T$
with $\{1,2,\ldots,p\}$, so that children are labeled with numbers
bigger than those of their ancestors.  Let $E(T)$ denote the set of
linear extensions of $T$.  Also, let us say that an internal node
$v$ is {\it to the left} (resp. {\it to the right}) of an internal
node $w$ if $v$ belongs to the left (resp. right) branch and $w$
belongs to the right (resp. left) branch of their closest (youngest)
common ancestor.

If $e$ is a linear extension of $T \in \T_p$, then we may define a
permutation $w_e \in S_p$ by reading the labels of the internal
nodes from left to right.  It is well known (see for example
\cite{LR}) that as $T$ varies over $\T_p$ and $e$ varies over $E(T)$
we obtain every $w \in S_p$ exactly once in this way.  For example,
the following are the three linear extensions of the same tree:
$$
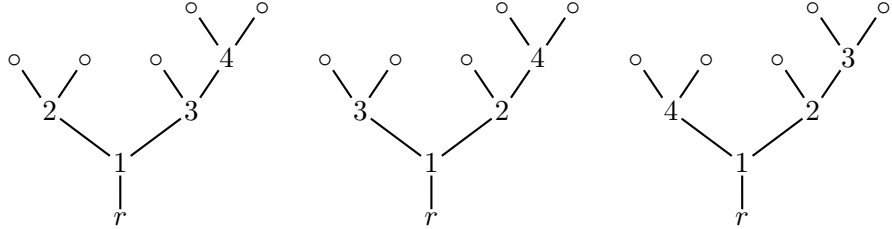
\pstree[treefit=tight,nodesep=1.2pt,levelsep=0.7cm,treemode=U]{\TR{r}}{
    \pstree{\TR{1}}{
        \pstree{\TR{2}}{
            \TR{\circ}\TR{\circ}
        }
        \pstree{\TR{3}}{
            \TR{\circ}
            \pstree{\TR{4}}{
                \TR{\circ}
                \TR{\circ}
            }
        }
    }
}\ \ \ \ \
\pstree[treefit=tight,nodesep=1.2pt,levelsep=0.7cm,treemode=U]{\TR{r}}{
    \pstree{\TR{1}}{
        \pstree{\TR{3}}{
            \TR{\circ}\TR{\circ}
        }
        \pstree{\TR{2}}{
            \TR{\circ}
            \pstree{\TR{4}}{
                \TR{\circ}
                \TR{\circ}
            }
        }
    }
} \ \ \ \ \
\pstree[treefit=tight,nodesep=1.2pt,levelsep=0.7cm,treemode=U]{\TR{r}}{
    \pstree{\TR{1}}{
        \pstree{\TR{4}}{
            \TR{\circ}\TR{\circ}
        }
        \pstree{\TR{2}}{
            \TR{\circ}
            \pstree{\TR{3}}{
                \TR{\circ}
                \TR{\circ}
            }
        }
    }
}
$$
corresponding to the permutations $2134, 3124$, and $4123$.

\begin{lem}
Let $T \in \T_p$.  Then
$$
f_{\Tree}^T = \sum_{e \in E(T)} q^{\inv(w_e)} \ \ \ \ \ {\rm and} \
\ \ \ \ f_{\Tree'}^T = 1.
$$
\end{lem}
\begin{proof}
The claim for $\Tree'$ is clear.  For $\Tree$, we will describe a
bijection between $E(T)$ and paths from $\emptyset$ to $T$.

Let $e'$ be a linear extension of $T'$ and suppose that $T = T_1
\vee T_2$ is obtained from grafting a splice of $T'$.  We may treat
$T_1$ and $T_2$ as subtrees of $T'$, and in particular restrict $e'$
to $T_1$ and $T_2$.  Thus we may define a labeling $e$ (depending on
$e'$, $T'$, $T_1$, and $T_2$) of $T$ by declaring it to be equal to
$e' + 1$ on $T_1 \cup T_2$, and equal to 1 on the new internal node
present in $T$ but absent in $T'$.  It is straight forward to see
that $e \in E(T)$.  Conversely, given $e \in E(T)$, it is easy to
recover $e'$ and $T'$ by comparing the labels along the leftmost
branch of $T_2$ with the labels along the rightmost branch of $T_1$.

Recursively applying this procedure we obtain the desired bijection
between $E(T)$ and paths from $\emptyset$ to $T$.  Finally, the
number of new inversions created in each step of this procedure is
equal to the number of internal nodes of $T_1$, which in turn is the
exponent of $q$ in $m(T',T)$.  This completes the proof.
\end{proof}

Thus Theorem 2 for $(\Tree, \Tree')$ amounts to grouping together
the terms of the left hand side of \eqref{eq:inv} into Catalan
number many terms.

\medskip
{\bf Acknowledgements.} I am grateful to Nantel Bergeron and Huilan
Li for related collaboration and conversations.

\end{document}